\documentclass[12pt,twoside]{amsart}
\usepackage{amsmath}
\usepackage{amsthm}
\usepackage{amsfonts}
\usepackage{amssymb}
\usepackage{latexsym}
\usepackage{mathrsfs}
\usepackage{amsmath}
\usepackage{amsthm}
\usepackage{amsfonts}
\usepackage{amssymb}
\usepackage{latexsym}
\usepackage{geometry}
\usepackage{dsfont}
\usepackage[dvips]{graphicx}
\usepackage[colorlinks=true,linkcolor=red,citecolor=blue]{hyperref}
\usepackage{color}
\usepackage[all]{xy}

\date{}
\allowdisplaybreaks[4] \footskip=15pt
\renewcommand{\uppercasenonmath}[1]{}

\numberwithin{equation}{section} \theoremstyle{plain}
\newtheorem*{thm*}{Main Theorem}
\newtheorem{thm}{Theorem}[section]
\newtheorem{cor}[thm]{Corollary}
\newtheorem*{cor*}{Corollary}
\newtheorem{lem}[thm]{Lemma}
\newtheorem*{lem*}{Lemma}
\newtheorem{prop}[thm]{Proposition}
\newtheorem*{prop*}{Proposition}
\newtheorem{rem}[thm]{Remark}
\newtheorem*{rem*}{Remark}
\newtheorem{exa}[thm]{Example}

\newtheorem*{exa*}{Example}
\newtheorem{df}[thm]{Definition}
\newtheorem*{df*}{Definition}

\newtheorem{con}[thm]{Construction}
\newtheorem*{con2*}{Construction B}
\newtheorem*{ack*}{ACKNOWLEDGEMENTS}

\def\m{\frak m}
\def\fd{{\rm fd}}

\newcommand{\pf}{\noindent\begin {proof}}
\newcommand{\epf}{\end{proof}}

\newcommand{\Ker}{\mbox{\rm ker}}
\newcommand{\Ext}{\mbox{\rm Ext}}
\newcommand{\Hom}{\mbox{\rm Hom}}
\newcommand{\Tor}{\mbox{\rm Tor}}
\newcommand{\im}{\mbox{\rm im}}
\newcommand{\coker}{\mbox{\rm coker}}
\def\lim{\mathop{\underrightarrow{\rm lim}}\nolimits}
\def\llim{\mathop{\underleftarrow{\rm lim}}\nolimits}

\newcommand{\spli}{\operatorname{spli}}
\newcommand{\silp}{\operatorname{silp}}
\newcommand{\sfli}{\operatorname{sfli}}

\newcommand{\op}{^{\mathrm{op}}}

\begin{document}
\begin{center}
{\large  \bf Mittag-Leffler Conditions, Gorenstein Modules and Homological Invariants}
	
	\vspace{0.5cm} \ Guocheng Dai$^{a}$,\    Xiaolei Zhang$^{b,\dag}$\\

	{\footnotesize a. \ School of Mathematics Sciences, Sichuan Normal University,  Chengdu 610066, China\\
			
		b.  School of Mathematics and Statistics,  Tianshui Normal University,  Tianshui 741001, China

		$\dag$.\ Corresponding author. E-mail addresses: (Xiaolei Zhang) zxlrghj@163.com\\}
\end{center}

\bigskip
\centerline { \bf  Abstract}
\bigskip
\leftskip10truemm \rightskip10truemm \noindent

In this paper, we investigate certain properties on the Mittag-Leffler conditions via set-theoretic methods.
We establish that a strongly $\aleph_1$-presented module $M$ satisfying $\Ext_R^{\ge 1}(M,D^{(\aleph_1)})=0$
belongs to the left orthogonal class of $\overline{D}$, where $\overline{D}$ is the definable closure of $D$.
This yields the consequence that every $\aleph_1$-generated strongly Gorenstein projective module is Gorenstein flat.
Furthermore, we investigate when a flat Gorenstein projective module is projective.
Finally,  for any two-sided $\aleph_1$-coherent ring $R$, we prove the identity
$\silp R + \silp R^{\mathrm{op}} = \spli R + \spli R^{\mathrm{op}}$. \\
\vbox to 0.3cm{}\\
{\it Key Words:} Mittag-Leffler condition; Gorenstein projective module; Gorenstein flat module; Homological invariant.\\
{\it 2020 Mathematics Subject Classification:} 16E10, 16E30.

\leftskip0truemm \rightskip0truemm
\bigskip

\section*{Introduction}

Throughout this paper, all rings are assumed to be associative rings with identity, and all modules are left modules unless otherwise stated.
Let $R$ be a ring. We denote by $R^{\mathrm{op}}$ its opposite ring, and we do not distinguish between right $R$-modules and left $R^{\mathrm{op}}$-modules.
Let $\mathcal{P}$ (resp. $\mathcal{FL}$) denote the class of projective (resp. flat) modules.

The Mittag-Leffler condition for inverse systems was originally introduced by Grothen-dieck in his work \cite{Gro1,Gro2}
on sheaf cohomology and the exactness of inverse limits.
In \cite{Gro1}, Grothendieck studied derived functors of $\varprojlim$ and
identified a criterion ensuring $\varprojlim^1 = 0$, later named the Mittag-Leffler condition.
The name `condition de Mittag-Leffler' \cite{Bo} was formally adopted and popularized by Bourbaki,
drawing an analogy with the classical Mittag-Leffler theorem on domains of convergence in complex analysis.
Let $\mathcal{C}$ be an abelian category.
Roos \cite{Ro} showed that
the first derived functor of inverse limit
vanishes on countable Mittag-Leffler sequences in $\mathcal{C}$ provided that $\mathcal{C}$ satisfying the Grothendieck axioms
$\textup{AB}3$ and $\textup{AB}4^*$ and having a set of generators. The examples given by Deligne \cite{De} and Neeman
\cite{Ne} show that the condition that the category has a set of generators is necessary.
The notion was later generalized to module theory as Mittag-Leffler modules in
the work of Raynaud and Gruson \cite{RA} as a tool for studying flatness and projectivity.

The findings of these scholars demonstrate that, under specified conditions,
the Mittag-Leffler condition is both necessary and sufficient for the vanishing
of $\varprojlim^1$ for a countable inverse system.
However, it remains an open problem to determine when the derived functors vanish in any uncountable case.
This desirable property in the countable setting ensures the technical feasibility of our approach.
The key technical result Proposition \ref{prop: 2.4} states that:
For a module $D$ and a strongly $\aleph_1$-presented module $M$,
$M$ belongs to the left orthogonal complement of the definable closure of $D$ provided that
$\operatorname{Ext}_R^{\ge 1}(M,D^{(\aleph_1)})=0$.
This result serve as a sufficient condition for determining whether $M$ is strictly stationary.
Precisely, based on the vanishing of $\varprojlim^1$,
the proof employs a sophisticated construction \ref{con: C} that lifts a short
exact sequence to a $\aleph_1$-continuous directed system of such sequences whose
bonding maps are relatively injective on the appropriate module. Lemma \ref{lem: 2.3}
employs stationary set combinatorics to extract a closed unbounded subset
on which the associated inverse system is continuous.
This set-theoretic argument is crucial for obtaining the orthogonality result.


Parallel to these developments, Gorenstein homological algebra emerged. Gorenstein projective, injective, and flat modules were introduced by Enochs, Jenda, and collaborators \cite{EJ1,EJa,EJT} as generalizations of classical homological modules. One central open problem, posed by Holm \cite{HO1}, asks whether every Gorenstein projective module is Gorenstein flat. Partial answers have been obtained via set-theoretic methods, particularly through the work of  \v{S}aroch and \v{S}\'{t}ov\'{\i}\v{c}ek \cite{SS2, SS1}, who employed $\kappa$-continuous direct systems and $<\kappa$-presented modules to transfer properties from small approximations to large modules. Another open problem, raised by  Bazzoni et al. \cite{BA}, asks whether every flat Gorenstein projective module is projective. They gave a positive answer under a condition of finiteness of a certain dimension. To date, no further progress has been made on this problem.

As a direct consequence of our  Proposition \ref{prop: 2.4}, we obtain that
every $\aleph_1$-generated strongly Gorenstein projective module is Gorenstein flat for any ring.
In particular, every $\aleph_1$-generated Gorenstein projective module over $\aleph_1$-coherent rings is Gorenstein flat,
thereby providing positive partial answers to Holm's question under a size restriction (see Proposition \ref{prop: 3.1}).
Compared with prior work \cite[Theorem 3.4]{WL}, our results represent a further advancement, strengthening our confidence in the validity of the general conclusion. Furthermore, we verify under an alterative condition when a flat Gorenstein projective module is projective.
Proposition \ref{prop: 6.1} shows that a flat Gorenstein projective module that is a direct limit of projectives over an index set of cardinality $<\aleph_\omega$ is projective.
We also verify the conditions under which a  Gorenstein projective module is Gorenstein flat (see  Proposition \ref{prop: 4.8}).

The final part of the paper studies homological invariants $\operatorname{silp}R$ (injective length of projectives) and $\operatorname{spli}R$ (projective length of injectives). For two-sided $\aleph_1$-coherent rings (where every $\aleph_1$-generated ideal is $\aleph_1$-presented),
Theorem \ref{silp+} establishes that $\spli R$ is finite if and only if $\silp R$ is finite, in which case they coincide.
Furthermore, the relation $\silp R + \silp R^{\mathrm{op}} = \spli R + \spli R^{\mathrm{op}}$ holds.
It follows that $\silp R=\spli R$ for all $\aleph_1$-coherent rings $R$  which is isomorphic with its opposite $R^{\mathrm{op}}$ (e.g., commutative $\aleph_1$-coherent rings). This generalizes previous work  by Emmanouil-Talelli ($\aleph_0$-Noetherian rings),  Chatzistavridis-Emmanouil (weakly coherent rings), and Wang-Yang (generalized coherent rings) (see \cite{CE25,E11,WY25}).
Finally, we consider a commutative local ring with $\mathfrak{m}^2=0$ and  infinite $\dim_k\mathfrak{m}=\kappa$,
we prove that $R$ is a $\kappa$-coherent ring  but fails to be weakly coherent,  and thus also not generalized coherent (see Proposition  \ref{thm:main}), demonstrating the novelty and  non-triviality of our findings.


\section {\bf Techniques on Mittag-Leffler conditions}

Angeleri H\"ugel and Herbera \cite{A08} introduced the notion of a module being $\mathcal{B}$-stationary.
These conditions on modules were successfully employed in a number
of different problems ranging from tilting theory to commutative algebra,
and even to a conjecture originating in algebraic topology.
The translation of certain homological properties of modules into Mittag-Leffler conditions was a key step
in solving some algebraic problems.
The  motivation for this section stems from the
results in \cite{BH, Sa, SS2}, where it
was made apparent that, for a module $M$, the vanishing of $\Ext^1_R(M,B)$
for all modules $B$ belonging to a class $\mathcal{B}$ closed under direct sums
can be characterized in terms of (strict) $\mathcal{B}$-stationarity.
We show that stationarity is a part of a common framework that works for strongly  $\aleph_1$-presented modules.
Let's start with the following basic concepts regarding  Mittag-Leffler conditions. Following \cite{A08, RA}

\begin{df}\label{df: 3.0} \textup{Let $B$ be a module.}

 \begin{enumerate}
\item \textup{An inverse system of modules $(H_{i}, h_{ij})_{i\leq j\in I}$ is said to satisfy the
\emph{Mittag-Leffler condition} if for any $i\in I$ there exists $j \geq i$
such that $h_{ik}(H_{k})=h_{ij}(H_{j})$ for any $k\geq j$.}

\item \textup{A direct system $(F_i,~f_{ji})_{i \leq j \in I}$
 of modules is said to be $B$-\emph{stationary} if the inverse system
$(\Hom_{R}(F_{i}, B), \Hom_{R}(f_{ji}, B))_{i \leq j\in I}$ satisfies the Mittag-Leffler condition.}

\item \textup{For a module $M$ and a class $\mathcal{B}$ of left $R$-modules,  we say that $M$ is $\mathcal{B}$-\emph{stationary} provided
that for some direct system $\mathcal{F}=(F_{i}, f_{ji}\mid i<j \in I)$
consisting of finitely presented modules with $\lim \mathcal{F}=M$, the corresponding inverse
system $\Hom_{R}(\mathcal{F},B)$ of abelian groups satisfies the Mittag-Leffler condition for each
$B\in \mathcal{B}$.}
\item \textup{Moreover, if we denote by $f_{i}: F_{i}\rightarrow M$
the canonical colimit map, we say that $M$ is strict $\mathcal{B}$-stationary if moreover we have {the stabilized term
$\im(\Hom_{R}(f_{ji}, B))= \im(\Hom_{R}(f_{i}, B))$ for each
$B\in \mathcal{B}$.}.}
\end{enumerate}
\end{df}

The class of $R$-modules which are strict $\mathcal{B}$-Mittag-Leffler  is closed under direct sums and direct summands and contains all projective modules (see \cite{A08}). The following result (a slightly more general version of it) is proved in \cite{Sa}.

\begin{lem}\label{ML}\cite[Lemma 4.2]{Sa}
Let $\mathcal{B}$ be a class of $R$-modules, which is closed under direct limits and direct products. Then, any $R$-module in the left orthogonal $^{\perp}\mathcal{B}$ is strict $\mathcal{B}$-Mittag-Leffler.
\end{lem}

Strict Mittag-Leffler modules can also be described via a natural
transformation in the following way.
Let $B$ be an $R$-module and $G$ an abelian group.
The abelian group $\operatorname{Hom}_{\mathbb{Z}}(B,G)$ of all additive
maps from $B$ to $G$ carries a natural right $R$-module structure.
For any $R$-module $M$ we define an additive map
\[
\Phi = \Phi_{M} \colon
\Hom_{\mathbb{Z}}(B,G) \otimes_{R} M
\longrightarrow
\Hom_{\mathbb{Z}}(\Hom_{R}(M,B),\, G),  \tag{1}
\]
by setting, for $f \in \Hom_{\mathbb{Z}}(B,G)$ and $m \in M$,
\[
\Phi(f \otimes m)(g)=f(g(m))\qquad
(g \in \Hom_{R}(M,B)).
\]
The map $\Phi$ is clearly natural in $M$.
If $P_{*}\to M\to 0$ is a projective resolution of $M$, the naturality of
$\Phi$ yields an induced chain map
\[
\Phi_{P_{*}}\colon
\operatorname{Hom}_{\mathbb{Z}}(B,G)\otimes_{R}P_{*}
\longrightarrow
\operatorname{Hom}_{\mathbb{Z}}\!\bigl(\operatorname{Hom}_{R}(P_{*},B),\,G\bigr).
\]
Assuming now that $G$ is divisible and passing to homology, we obtain additive maps
\begin{equation}
	\Phi_{M}^{(n)}\colon
	\Tor_{n}^{R} (\Hom_{\mathbb{Z}}(B,G),M)
	\longrightarrow
	\Hom_{\mathbb{Z}} (\Ext_{R}^{n}(M,B),G),
	\qquad n\ge 0, \tag{2}
\end{equation}
which are independent of the chosen projective resolution of $M$.
For a fixed resolution we denote by $\Omega^{n}M$ the $n$-th syzygy module, i.e.
$\Omega^{n}M=\operatorname{im}(P_{n}\to P_{n-1})$.
A proof of the following characterization can be found in \cite[Proposition 1.5]{E11};
the argument relies on \cite[Theorem 8.11]{A08}.

\begin{prop}\label{nSML}
Let $\mathcal{B}$ be a class of $R$-modules and fix a non-negative integer $n$. Then, the following conditions are equivalent for an $R$-module $M$:
\begin{enumerate}
\item[(i)] The map (2) is injective for any $R$-module $B \in \mathcal{B}$ and any divisible abelian group $G$.
\item[(ii)] The $n$-th syzygy module $\Omega^{n}M$ of $M$ is strict $\mathcal{B}$-Mittag-Leffler.
\end{enumerate}
\end{prop}

\vspace{0.2cm}

Let $\{H_i, g_{ji}\mid i\leq j\in I\}$ be an  inverse system of $R$-modules. Following \cite{JE}, there is a complex
$${\xymatrix{0\ar[r]&\llim H_{i}\ar[r]&\prod_{i}H_{i}\ar[r]^(.4){\Delta_{0}}&\prod_{i_{0}<i_{1}}H_{i_{0}i_{1}}\ar[r]^(.45){\Delta_{1}}&
\prod_{i_{0}<i_{1}<i_{2}}H_{i_{0}i_{1}i_{2}}\ar[r]^(.7){\Delta_{2}}&\cdots,
}}$$
where $H_{i_{0}i_{1}\cdots i_{n}}=H_{i_{0}}$  for all $i_0<i_{1}<\cdots<i_{n}\in I$  and
\begin{align*}
\Delta_{0}((h_i)_{i\in I})=(h_{i_0}-g_{i_1i_0}(h_{i_0}))_{_{i_{0}<i_{1}}},\\
\Delta_{1}((h_{i_{0}i_{1}})_{i_{0}<i_{1}})=(h_{i_0i_2}-h_{i_0i_1}-g_{i_1i_0}(h_{i_1i_2}))_{i_{0}<i_{1}<i_{2}}.
\end{align*}
$\llim^n H_i=n$-th cohomology group of the complex, i.e., $\llim^{n}H_i=\frac{\Ker (\Delta_{n})}{\im (\Delta_{n-1})}.$

Let $\mathcal{M}=(F_{i},~f_{ji}:F_{i}\rightarrow F_{j}\mid i\leq j\in I)$
be a direct system of R-modules  with $M=\lim\mathcal{M}$. By the exactness of the direct limit there is an acyclic complex:
$${\xymatrix{\cdots\ar[r]^(.25){\delta_{2}}&
\bigoplus_{i_{0}<i_{1}<i_{2}}F_{i_{0}i_{1}i_{2}}\ar[r]^(.55){\delta_{1}}&\bigoplus_{i_{0}<i_{1}}F_{i_{0}i_{1}}
\ar[r]^(.55){\delta_{0}}&\bigoplus_{i}F_{i}\ar[r]&\lim F_{i}\ar[r]&0,
}}$$

where $F_{i_{0}i_{1}\cdots i_{n}}=F_{i_{0}}$  for all $i_0<i_{1}<\cdots<i_{n}$ in $I$  and
\begin{align*}
(\delta_{0}\upharpoonright F_{ij})(x)=(x,{-}f_{ji}(x))\in F_{i}\times F_{j},\\
~~~~(\delta_{1}\upharpoonright F_{ijk})(x)=(x,-x,{-}f_{ji}(x))\in F_{ik}\times F_{ij}\times F_{jk}.
\end{align*}
For a module $C$, applying the functor $\Hom_R(-,C)$ to the long exact sequence, we get a complex as that above.
For particular use, we give the following easy observations.

\begin{lem}\label{lem: 2.1}\begin{enumerate}
\item[(1)] $\llim^{n}H_i=0$ if and only if the complex is
exact at the $n$-th position.  In particular, $\llim^{1}H_i=0$  if and only if
for every family $(h_{ij})_{i<j}\in\prod_{i_{0}<i_{1}}H_{i_{0}i_{1}}$ satisfying $h_{ik}=h_{ij}+g_{kj}(h_{ij})$ for all $i<j<k\in I$,
there exists a family $(h_i)_{i\in I}\in\prod_{i\in I}H_{i}$ such that $h_i = h_{ij}+ g_{ji}(h_j)$ for all $i<j\in I$.
\item[(2)] $\llim^1 \Hom_R(F_i,C)=0$ if and only if
for every family of morphisms $(g_{ji} : F_i\rightarrow C \mid i<j)$ such that $g_{ki} = g_{ji} + g_{kj} f_{ji}$ for all $i<j<k$ in I,
there exists a family of morphisms $(g_i: F_i \rightarrow C \mid i\in I)$  such that $g_i = g_{ji} +g_j f_{ji}$ for all $i<j$  in I. In particular,  $\llim^{1}\Hom_R(F_i,C)=0$ provided that $\Ext^1_R(M,C)=0$.
 \end{enumerate}
\end{lem}

For a given homomorphism $f: M\rightarrow N$ and a module $C$,
we denote by $f^*=\Hom_{R}(f,C)$ for simplicity.

\vspace{0.2cm}

Let
$\mathcal{M}=(M_\alpha, g_{\beta\alpha} : M_\alpha\rightarrow M_\beta \mid \alpha<\beta\leq \aleph_1)$
be a continuous well-ordered direct system of modules, where {$M_\alpha$ is $<\aleph_1$-presented for each $\alpha<\aleph_1$ and} $\lim_{\alpha<\aleph_1} M_{\alpha}=M_{\aleph_1}=M$.
Let $B$ be a module and apply the functor $\Hom_R(-,B)$ to $\mathcal{M}$.
Let  $I_{\alpha}=\bigcap_{\alpha\leq\beta<\aleph_1}\im(g_{\beta\alpha}^{*})$.
Since $I_{\alpha}\subseteq \Hom_{R}(M_{\alpha},B)$,
we get that ${g_{\beta\alpha}^{*}\upharpoonright}_{I_{\beta}}:I_{\beta}\rightarrow I_{\alpha}$
is epic for each $\alpha\leq\beta<\aleph_1$.
We obtain an inverse system
$(I_{\alpha},{g_{\beta\alpha}^{*}\upharpoonright}_{I_{\beta}}:I_{\beta}\twoheadrightarrow I_{\alpha}\mid \alpha\leq \beta<\aleph_1)$.
For each limit ordinal $\beta<\aleph_1$, taking the inverse limit indexed on the initial section $\beta$,
$${\xymatrix{
0\ar[r]&I_\beta\ar[d]^{\subseteq}\ar[r]^(.40){\subseteq}
&\Hom_{R}(M_{\beta},B)\ar[d]^{=}\\
0\ar[r]&\llim_{\alpha<\beta} I_\alpha\ar[r]^(.40){\subseteq}
& \Hom_{R}(M_{\beta},B).}}$$
Let $J_\beta=I_\beta$ for $\beta$ successional, and $J_\beta=\llim_{\alpha<\beta} I_\alpha$ for $\beta$ is  limit.
Identifying the connecting homomorphisms ${g_{\beta\alpha}^{*}\upharpoonright}_{J_{\beta}}:J_\beta\rightarrow J_\alpha$.
Unambiguously, we denote
${g_{\beta\alpha}^{*}\upharpoonright}_{J_{\beta}}$ by $g_{\beta\alpha}^{*}$.
We get the following  Mittag-Leffler inverse system
$${\xymatrix{\mathcal{J}=(J_{\alpha},g_{\beta\alpha}^{*}:J_{\beta}\ar[r]& J_{\alpha}\mid \alpha\leq \beta<\aleph_1).}}\eqno {(2.1)}$$

With the notation above, for this inverse system (2.1), we have the following lemma.

\begin{lem}\label{lem: 2.2}
Assume $\llim^1 \Hom_R(M_\alpha, B)=0$. Given any
family of morphisms $(\varphi_{\beta\alpha} : M_\alpha\rightarrow B \mid \alpha<\beta)$ satisfying
$\varphi_{\gamma\alpha}=\varphi_{\beta\alpha}+\varphi_{\gamma\beta}g_{\beta\alpha}$ and $\varphi_{\beta\alpha}\in J_\alpha$
for all $\alpha<\beta<\gamma<\aleph_1$. Then
there exists a family of morphisms $(\varphi_\alpha\mid \alpha<\aleph_1)$ with each $\varphi_\alpha\in J_{\alpha}$,
such that $\varphi_\alpha = \varphi_{\beta\alpha} +\varphi_\beta g_{\beta\alpha}$ for all $\alpha<\beta<\aleph_1$.
\end{lem}
\begin{proof}By assumption that $\llim^1 \Hom_R(M_\alpha, B)=0$, for the given compatible family $(\varphi_{\beta\alpha})_{\alpha<\beta}$,
by Lemma \ref{lem: 2.1}(2),
there exists a family of morphisms $(\varphi_\alpha\mid \alpha<\aleph_1)$ with each ${\varphi}_\alpha\in \Hom_R(M_\alpha, B)$,
such that $$\varphi_\alpha = \varphi_{\beta\alpha} +\varphi_\beta g_{\beta\alpha},~~ ~~~~~~~~~\forall~~~\alpha<\beta<\aleph_1.$$
To show that $\varphi_\alpha\in J_{\alpha}$ for each $\alpha<\aleph_1$, we
subtract the relation $\varphi_{\gamma\alpha}=\varphi_{\beta\alpha}+\varphi_{\gamma\beta} g_{\beta\alpha}$ from the equation above,
which yields the equivalent identity
$$\varphi_\alpha -\varphi_{\gamma\alpha} = (\varphi_\beta -\varphi_{\gamma\beta}) g_{\beta\alpha}, ~~~~~~~\forall~~~ \alpha<\beta<\gamma<\aleph_1.$$
It follows that the difference $\varphi_\alpha -\varphi_{\gamma\alpha}$ lies in the intersection of the images of the maps
$g_{\beta\alpha}^{*}$:
$$\varphi_\alpha -\varphi_{\gamma\alpha}\in \bigcap_{\beta>\alpha}\im(g_{\beta\alpha}^{*})\subseteq J_\alpha$$ {for all $\alpha<\beta<\gamma<\aleph_1$}.  As
${\varphi}_{\gamma\alpha}\in J_\alpha$ and $J_\alpha$ is a submodule, the inclusion above forces $\varphi_\alpha\in J_\alpha$.
\end{proof}

Recall from \cite[Definition 6.34]{GTT} that a well ordered inverse system $\mathcal{N}=\{N_\alpha, g_{\alpha\beta}: N_{\beta}\rightarrow N_{\alpha}~|~\alpha<\beta\leq\sigma \}$ is called \emph{continuous} if the connecting maps $g_{\alpha\beta}$ are epimorphisms and
$N_{\theta}= \llim_{\alpha<\theta}N_{\alpha}$ for each  $\theta\leq\sigma$.
Correspondingly, we call a well ordered inverse system $\mathcal{N}=\{N_\alpha, g_{\alpha\beta}: N_{\beta}\rightarrow N_{\alpha}~|~\alpha<\beta<\sigma \}$ \emph{inner-continuous} provided that the connecting maps $g_{\alpha\beta}$ are epimorphisms and
$N_{\theta}= \llim_{\alpha<\theta}N_{\alpha}$ for each {limit ordinal} $\theta<\sigma$.


\begin{lem}\label{lem: 2.3}Set $B=D^{(\aleph_1)}$ and assume $\llim^1 \Hom_R(M_\alpha, B)=0$ {with $M_\alpha$ $<\aleph_1$-presented for each $\alpha<\aleph_1$}.
There exists a closed unbounded subset $X \subseteq \aleph_1$ such that the inverse subsystem of $\mathcal{J}$
indexed on $X$ is inner-continuous.
\end{lem}
\begin{proof}Based on the construction of the inverse system $\mathcal{J}$ in (2.1), we may
restrict to the subsystem indexed on the closed unbounded subset of $\aleph_1$ consisting of limit ordinals in $\aleph_1$.
 Without loss of generality,  we may assume that whenever $g_{\beta\alpha}^{*}$ is not epic for $\alpha < \beta$,
then $g_{\alpha+1,\alpha}^{*}$ is already  not epic.
		
Suppose, for the sake of contradiction, that the set\[E = \{\alpha < {\aleph_1} \mid g_{\alpha+1,\alpha}^{*} \text{ is not  epimorphic}\}\]
is stationary in $\aleph_1$; that is, it intersects every closed unbounded subset of $\aleph_1$.
Fix $h_\alpha \in J_\alpha\subseteq\operatorname{Hom}_R(M_\alpha, B)$ such that $h_\alpha$
does not factorize through $g_{\alpha+1,\alpha}$ for each $\alpha \in E$. Define $h_\alpha = 0$ for $\alpha \in {\aleph_1} \setminus E$.
		
We will inductively construct homomorphisms $\varphi_{\beta\alpha} : M_\alpha \rightarrow B^{(\beta)}$ for all $\alpha < \beta \leq \aleph_1$ such that
$\varphi_{\alpha+1,\alpha}$ is the composition of $h_\alpha$ with the $\alpha$-th canonical inclusion $B\rightarrow B^{(\alpha+1)}$, and
$\varphi_{\gamma\alpha} = \varphi_{\beta\alpha} + \varphi_{\gamma\beta} g_{\beta\alpha}$ for all $\alpha < \beta < \gamma \leq\aleph_1$.

For $\beta =1$, put $\varphi_{10} = h_0$.  Assume
$\varphi_{\beta\alpha}$ has been constructed for all $\alpha < \beta < \gamma$ for some $\gamma \leq \theta$.
Suppose $\gamma$ is a successor, say $\gamma = \delta + 1$.
Set $\varphi_{\gamma\delta} = h_\delta$ and {$\varphi_{\gamma\alpha} = \varphi_{\delta\alpha} + \iota_\delta h_\delta g_{\delta\alpha}$} for each $\alpha < \delta$.
Suppose $\gamma<\aleph_1$ is a limit ordinal, then
$\llim^1_{\gamma<\alpha}J_\gamma=0,~ \alpha<\aleph_1$.  Thus, we may
define $\varphi_{\gamma\alpha}$ as the maps $\varphi_\alpha$ given by Lemma \ref{lem: 2.1} (1). For the case
$\gamma=\aleph_1$, given the collection  maps $(\varphi_{\beta\alpha})_{\alpha<\beta<\aleph_1}$  already constructed,
and note the assumption that $\llim^1 \Hom_R({M}_\alpha,D^{(\aleph_1)})=0$.
By Lemma \ref{lem: 2.2},
there exists a family of morphisms $(\varphi_\alpha\mid \alpha<\aleph_1)$ with each $\varphi_\alpha\in J_{\alpha}$,
such that $\varphi_\alpha = \varphi_{\beta\alpha} +\varphi_\beta g_{\beta\alpha}$ for all $\alpha<\beta<\aleph_1$.
Hence,
we define $\varphi_{\aleph_1\alpha}$ as the maps $\varphi_\alpha$.

Next, employing the same technique as in the proof of \cite[Theorem 2.9]{SS2},
we put
$${\xymatrix{Z = \{\lambda < \aleph_1 \mid \operatorname{im} (\varphi_{\aleph_1\alpha}) \subseteq B^{(\lambda)} \text{ for any } \alpha < \lambda<\aleph_1\}.}}\eqno {(2.2)}$$
It is straightforward to verify that $Z$ is closed unbounded in $\aleph_1$ (see Remark \ref{rem: 1.7} below).
Consequently, the set $Z^c$ consisting of
the limit ordinals in $Z$ is also closed unbounded. Therefore,  there exists some $\lambda \in Z^c \cap E$.
For $\lambda<\aleph_1$,	denote by $\pi$ the $\lambda$-th canonical projection $B^{(\aleph_1)} \rightarrow B$.
First, we show that $\pi \varphi_{\aleph_1\lambda} = 0$. Choose an arbitrary $x \in {M_\lambda}$.
Since {$M_\lambda = \lim_{\alpha < \lambda} M_\alpha$} by continuity of the direct system, there exists
$\alpha < \lambda$ and $y \in {M_\alpha}$ such that $x = g_{\lambda\alpha}(y)$. We have the equality:
	$$\pi \varphi_{\aleph_1\alpha}(y) = \pi \varphi_{\lambda\alpha}(y) + \pi \varphi_{\aleph_1\lambda} g_{\lambda\alpha}(y).$$		
However, $\pi \varphi_{\aleph_1\alpha}(y) = 0$ since $\lambda \in Z$ and $\pi \varphi_{\lambda\alpha}(y) = 0$ by definition of $\varphi_{\lambda\alpha}$.
Hence $0 = \pi \varphi_{\aleph_1\lambda} g_{\lambda\alpha}(y) = \pi \varphi_{\theta\lambda}(x)$. The claim follows since $x \in {M_\lambda}$ was arbitrary.	
	On the other hand, we know that $\varphi_{\aleph_1\lambda} = \varphi_{\lambda+1,\lambda} + \varphi_{\aleph_1,\lambda+1} g_{\lambda+1,\lambda}$. Composing this with $\pi$, we obtain:
	$$0 = \pi \varphi_{\aleph_1\lambda} = \pi \varphi_{\lambda+1,\lambda} + \pi \varphi_{\aleph_1,\lambda+1} g_{\lambda+1,\lambda} =h_\lambda + \pi \varphi_{\aleph_1,\lambda+1}g_{\lambda+1,\lambda}.$$	
	But this implies that $h_\lambda$ factorizes through $g_{\lambda+1,\lambda}$, a contradiction to the choice of $h_\lambda$.	
	Therefore, $E$ is not stationary. Consequently, we can choose a closed unbounded subset $X \subseteq \theta$ such that $X \cap E = \emptyset$
and the inverse subsystem of $\mathcal{J}$ indexed on $X$ is continuous.
\end{proof}

\begin{rem}\label{rem: 1.7}
In set theory, an unbounded closed subset of a regular cardinal--commonly referred to as a \textup{club}-- trivially
 satisfies the required property for the set  $Z$  defined in \textup{(2.2)}.
Nevertheless, for the reader convenience, we sketch the brief verification.

{Let $\theta$ be a given regular cardinal.} In the definition of $X$,
the symbol $\lambda= \{\alpha| \alpha<\lambda\}$ denotes an ordinal number below $\theta$, and is not necessarily
a cardinal. Since each $M_\alpha$ ($\alpha<\theta$)
 is $<\theta$-generated and $\theta$ is regular, for each $\delta<\theta$ there exists some
$\delta\leq\delta'<\theta$  such that $\im g_{\theta\alpha}\subseteq D^{(\delta')}$ for all $\alpha<\delta$, where
$g_{\theta\alpha}:M_\alpha\rightarrow D^{(\theta)}$.
Now, starting from an arbitrary $\lambda_0<\theta$, define recursively $\lambda_{n+1}=\lambda'_{n}$ for $n <\omega\}$.
Setting $\lambda=\textup{sup}\{\lambda_n | n <\omega\} < \theta$,
the regularity of $\theta$ ensures that $\lambda<\theta$. We claim that
$\lambda\in X$. Indeed, given an $\alpha<\lambda$, we find an $n<\omega$ such that $\alpha<\lambda_n$, whence
$\im g_{\theta\alpha}\subseteq D^{(\lambda_{n+1})}\subseteq D^{(\lambda)}$. This shows that $X$ is unbounded in $\theta$,
and the closedness is equally straightforward
\end{rem}

Next, we require the following two technical constructions concerning direct systems.

\begin{con}\rm{(}\cite[Page 37]{SS2}\rm{)}\label{con: A} Let $f: M \rightarrow N$ be a homomorphism of modules, $\kappa$ a regular
uncountable cardinal, $\mathcal{M}=(M_{i}, f_{ji}: M_{i}\rightarrow M_{j}\mid i<j\in I)$
and $\mathcal{N}=(N_{i}, g_{ji}:~ N_{i}\rightarrow N_{j}\mid~ i<j\in J)$
$\kappa$-continuous direct systems of $<\kappa$-presented modules
such that ${\lim} \mathcal{M}=M$ and
${\lim} \mathcal{N}=N$. Then there is a $\kappa$-continuous direct system
$\mathcal{U}=(u_{k}:M_{i_{k}}\rightarrow N_{j_{k}}, (f_{i_{l},i_{k}} , g_{j_{l},j_{k}})\mid k < l \in K)$ consisting of morphisms with
domains in $\mathcal{M}$ and codomains in $\mathcal{N}$ such that ${\lim} \mathcal{U}=f$.
Subsequently, there is a $ \kappa$-continuous direct system $\mathcal{K}=(\coker(u_{k}) \mid k \in K)$
consisting of $<\kappa $-presented modules (with canonically defined maps).
\end{con}

\begin{con}\rm{(}\cite[Page 38]{SS2}\rm{)}\label{con: B}
Let $M\in R$-\textup{Mod}, $\kappa$ be a regular uncountable cardinal, $\gamma<\kappa$ and,
for each $\alpha\leq\gamma$, let
$\mathcal{M}^{\alpha}=(M^{\alpha}_{ji}, f^{\alpha}_{ji}:M^{\alpha}_{i}\rightarrow M^{\alpha}_{j} \mid i < j \in I_{\alpha})$
be a $\kappa$-continuous direct system consisting of $<\kappa$-presented modules such that
${\lim} \mathcal{M}^{\alpha}=M$.
Then the systems $\mathcal{M}^{\alpha}$, $\alpha\leq\gamma$, have a common cofinal $\kappa$-continuous subsystem.
More precisely: for each $\alpha\leq\gamma$, there exists a $\kappa$-continuous cofinal direct subsystem
$\mathcal{N}^{\alpha}=(M^{\alpha}_{ji}, f^{\alpha}_{ji}:M^{\alpha}_{i}\rightarrow M^{\alpha}_{j} \mid i < j \in J_{\alpha})$
of $\mathcal{M}^{\alpha}$; furthermore, for any $\alpha,~\beta\leq\gamma$,
there is a bijection $\iota: J_{\alpha}\rightarrow J_{\beta}$ and a direct system $\mathcal{U}$ with
${\lim} \mathcal{U}^{\alpha}=\textup{id}_{M}$ whose objects are isomorphisms
$u_{i}:M^{\alpha}_{i}\rightarrow M^{\beta}_{\iota_{i}}$, $i\in J_{\alpha}$,
and for each $i < j \in J_{\alpha}$, there is only one morphism from $u_{i}$ to $u_{j}$ in $\mathcal{U}$,
namely $( f^{\alpha}_{ji}, f^{\beta}_{\iota(j),\iota(i)})$.
\end{con}

\vspace{0.2cm}

Utilizing the two constructions above,
we construct a \emph{decomposition} of a given short exact sequence.
\begin{con}\rm{\label{con: C}
Let $\kappa$ be an uncountable regular cardinal and let $M$ be a  module and fit into
an exact sequence
$${\xymatrix{0\ar[r]&N\ar[r]^{u}&P\ar[r]^{v}&M\ar[r]&0}}\eqno {(2.3)}$$
which is exact after applying the functor $\Hom_{R}(-,B)$ for a module $B$.
By \cite[Lemma 1.1]{SS1}, there exist two $\kappa$-continuous
directed systems $\mathcal{M}=(M_{i},~f_{ji}\mid i\leq j\in X_{1})$
and $\mathcal{N}=(N_{i},~g_{ji}\mid i\leq j\in X_{2})$ consisting
of $<\kappa$-presented modules, such that
$M=\lim \mathcal{M}$ and $N=\lim \mathcal{N}$.
By Kaplansky's theorem, $P$ is the direct sum
of countably generated projective modules. Let $\mathcal{H}=(P_{i},~\epsilon_{ji}\mid i\leq j\in X_{3})$
be the directed system consisting of countably generated direct summands of
$P$ and inclusions. It obvious that  $P$ is the directed union of these submodules in $\mathcal{H}$.

We get a $\kappa$-continuous directed system
$\mathcal{W}=(N_{i}\buildrel {u_{i}}\over\rightarrow P_{i}\buildrel {v_{i}} \over\rightarrow M_{i}\rightarrow 0,~~g_{ji},e_{ji},f_{ji}\mid i<j\in X)$
such that
$\lim\mathcal{W}=(0\rightarrow N \buildrel {u }\over\rightarrow P \buildrel {v } \over\rightarrow M \rightarrow 0,~~g_{ i},e_{ i},f_{ i}\mid i \in X).$
Applying $\Hom_{R}(-,B)$ to the exact sequences
in $\mathcal{W}$ yields
the following inverse system $\mathcal{W}^{*}$ of exact sequences:
$$(0\rightarrow \Hom_{R}(M_{i},B)\buildrel{v_{i}^{*}}\over\rightarrow \Hom_{R}(P_i,B)
\buildrel {u_{i}^{*}}\over\rightarrow\Hom_{R}(N_i,B), f_{ji}^{*},e_{ji}^{*},g_{ji}^{*}\mid i\leq j\in X).$$
We consider the inverse limit, denoted by $\llim\mathcal{W}^{*}$:
$$(0\rightarrow \Hom_{R}(M,B)\buildrel {v^*}\over\rightarrow \Hom_{R}(P,B)
\buildrel {u^*}\over\rightarrow \Hom_{R}(N,B)\rightarrow 0, f^*_{i},e^*_{i},g^*_{i}\mid i\in X).\eqno {(2.4)}$$
Let  $L_{i}=\im(u_{i}^{*})$, $i\in X$.
Since $L_{i}\subseteq \Hom_{R}(N_{i},B)$), unambiguously, we denote
${g_{ji}^{*}\upharpoonright}_{I_{j}}:L_{j}\rightarrow L_{i}$ by $g_{ji}^{*}$. With these notation,
we have the following inverse system
$$(0\rightarrow \Hom_{R}(M_{i},B)\buildrel{v_{i}^{*}}\over\rightarrow \Hom_{R}(P_i,B)
\buildrel {u_{i}^{*}}\over\rightarrow L_i\rightarrow 0), f_{ji}^{*},
e_{ji}^{*},g_{ji}^{*}\mid i\leq j\in X).\eqno {(2.5)}$$
with inverse limit $\llim\mathcal{W}^{*}$.  The homomorphisms
$g_{ji}^{*}$, $i<j\in X$, are epimorphic since $u^*$ is epimorphic.\rm }
\end{con}
Under the above constructions and notations, we have the following lemma.

\begin{lem}\label{lem: 2.4}Let $M$  be a strongly $\aleph_1$-presented module  and $\Ext^{\geq 1}_R(M, D^{(\aleph_1)})=0$ for a given module $D$.
Let $\mathcal{M}=(M_{\alpha},~f_{\beta\alpha}\mid \alpha\leq \beta <\aleph_1)$ be a continuous direct system with {with $M_\alpha$ $<\aleph_1$-presented for each $\alpha<\aleph_1$ and} direct
limit $\lim\mathcal{M}=(M, f_\alpha \mid \alpha<\aleph_1)$.
There is a closed unbounded subset $S\subseteq \aleph_1$ such that
$f_{\beta\alpha}$ is $ D^{(\aleph_1)}$-injective for all $\alpha, \beta \in S$, $\alpha < \beta$.
\end{lem}
\begin{proof} Put $B=D^{(\aleph_1)}$.
By the preceding constructions, there exists a strongly $\aleph_1$-presented module $N$ such that
$M$ fits into the exact sequence (2.3).
Consequently, we obtain the inverse systems (2.5) indexed on $\aleph_1$ with limit (2.4) and
$$\mathcal{N}=(N_{\alpha},~g_{\beta\alpha}\mid \alpha\leq \beta <\aleph_1)~~~~ \textup{such that}~~~~
\lim\mathcal{N}=(N, g_\alpha \mid \alpha<\aleph_1).$$
By assumption, we have $\Ext^{1}_R(N, B)=0$ and thus $\llim^1 \Hom_R(N_\alpha, B)=0$. By the constructions of the inverse systems
(2.1) and (2.5), we get that $L_\alpha\buildrel{\subseteq}\over\rightarrow J_\alpha\buildrel{\subseteq}\over\rightarrow\Hom_R(N_\alpha, B)$.
Note that the transition maps are ${g_{\beta\alpha}^{*}\upharpoonright}_{J_{\beta}}:J_\beta\rightarrow J_\alpha$ and
${g_{\beta\alpha}^{*}\upharpoonright}_{L_{\beta}}:L_\beta\rightarrow L_\alpha$, for each $\alpha\leq \beta <\aleph_1$.
Taking the inverse limits, we get that {$\llim J_\alpha=\Hom_R(N, B)$}.
Let $J_{\aleph_1}=\Hom_R(N, B)$, by Lemma \ref{lem: 2.3}, we get a continuous inverse system
$$\mathcal{J}'=(J_{\alpha},g_{\beta\alpha}^{*}:J_{\beta}\rightarrow J_{\alpha}\mid \alpha\leq \beta\leq\aleph_1)~~~~ \textup{such that}~~~~
\lim\mathcal{J}'=(\Hom_R(N, B), g^*_\alpha \mid \alpha<\aleph_1).$$
Combining $\mathcal{J}'$, (2.4) and (2.5), we may form the following inverse system of short exact sequences:
$${\xymatrix{0\ar[r]&\Hom_{R}(M,B)\ar[d]\ar[r]^{v^{*}}&\Hom_{R}(P,B)\ar[d]
\ar@{->>}[r]^{u^{*}}&\Hom_{R}(N,B)\ar@{->>}[d]^{g^{*}_{\beta}}\ar[r]&0\\
0\ar[r]&\Hom_{R}(M_{\beta},B)\ar[d]^{f^{*}_{\beta\alpha}}\ar[r]^{v_{\beta}^{*}}
&\Hom_{R}(P_{\beta},B)\ar[d]\ar[r]^(.70){ u_{\beta}^{*}}&J_\beta\ar@{->>}[d]^{g^{*}_{\beta\alpha}}\ar[r]&0\\
0\ar[r]&\Hom_{R}(M_{\alpha},B)\ar[r]^{ v_{\alpha}^{*}}
& \Hom_{R}(P_{\alpha},B)\ar[r]^(.70){ u_{\alpha}^{*}}
&J_\alpha\ar[r]&0.}}\eqno {\raisebox{-8ex}{$(2.6)$}}$$

For each  limit ordinal $\beta\leq\aleph_1$, taking the inverse limit on the subsystem indexed on the initial $\beta$ segment, we have
the following exact sequence
$$0\rightarrow \llim_{\alpha<\beta}\Hom_{R}(M_{\alpha},B)\rightarrow \llim_{\alpha<\beta}\Hom_{R}(P_\alpha,B)\buildrel{u^{*}}\over\rightarrow$$
$$\rightarrow \llim_{\alpha<\beta} J_\alpha\rightarrow
\llim_{\alpha<\beta}^1\Hom_{R}(M_{\alpha},B)\rightarrow \llim_{\alpha<\beta}^1\Hom_{R}(P_\alpha,B)\rightarrow \cdots.$$
By  the continuity of $\mathcal{J}'$,  $\llim_{\alpha<\beta} J_\alpha=J_\beta$.
Since $\llim_{\alpha<\beta}^1\Hom_{R}(P_\alpha,B)=0$,
we get that $\llim_{\alpha<\beta}^1\Hom_{R}(M_{\alpha},B)=0$ for every $\beta\leq \aleph_1$ as $u^{*}$ is an epimorphism.

Combining Lemma \ref{lem: 2.1} (2), and adopting an approach  similar to that in the proof of  Lemma \ref{lem: 2.2},
we can choose a closed unbounded subset $S \subseteq \aleph_1$ such that
and $f_{\beta\alpha}$ is $D^{(\aleph_1)}$-injective for all $\alpha, \beta \in S$, $\alpha < \beta$.
\end{proof}

We say that a  class $\mathcal{D}\subseteq$ $R$-$\textup{Mod}$ is a \emph{definable} \cite{HD,SS1} if
$\mathcal{D}$ is closed under products, direct limits and pure submodules.
For a class $\mathcal{C}\subseteq \textup{Mod}$-$R$, we denote  by $\overline{\mathcal{C}}$
the definable closure of the class $\mathcal{C}$, that is, the smallest definable class containing $\mathcal{C}$, and
$\textup{Cogen}_{*}(\mathcal{C})$ the class of
all pure submodules of products of modules from $\mathcal{C}$.
If $\mathcal{C}=\{C\}$, we write just $\textup{Cogen}_{*}(C)$, $\overline{C}$, respectively.
It is easy to see that $\overline{\mathcal{P}}=\overline{R}$.

We have the main result of this section, and we will give some applications in the following sections.

\begin{prop}\label{prop: 2.4} Let $M$  be a strongly $\aleph_1$-presented module  and $\Ext^{\geq 1}_R(M, D^{(\aleph_1)})=0$.
Then $M\in {^\perp \overline{D}}$, where $\overline{D}$ is the definable closure  of $D$.
\end{prop}
\begin{proof}Since $M$  is  strongly $\aleph_1$-presented, it fits into the exact
sequence {(2.3)} with $N$ strongly $\aleph_1$-presented and $\Ext^{\geq 1}_R(N, D)=0$.
Applying Construction \ref{con: C}, Lemmas \ref{lem: 2.2} and \ref{lem: 2.3} to both $M$ and $N$,
we obtain an $\aleph_1$-continuous directed system
$\mathcal{W}=(N_{\alpha}\buildrel {u_{\alpha}}\over\rightarrow P_{\alpha}\buildrel {v_{\alpha}} \over\rightarrow M_{\alpha}\rightarrow 0,~~g_{\beta\alpha},e_{\beta\alpha},f_{\beta\alpha}\mid \alpha\leq\beta< \aleph_1)$
such that
${\lim}\mathcal{W}=(0\rightarrow N \buildrel {u }\over\rightarrow P \buildrel {v } \over\rightarrow M
\rightarrow 0,~~g_{\alpha},e_{\alpha},f_{\alpha}\mid \alpha\in \aleph_1 ),$
where the homomorphisms in the connecting triples $(g_{\beta\alpha})_{\alpha<\beta}$, $(e_{\beta\alpha})_{\alpha<\beta}$ and
$(f_{\beta\alpha})_{\alpha<\beta}$ are all $ D^{({\aleph_1})}$-injective. Thus the limit maps
$g_{\alpha}$, $e_{\alpha}$ and $f_{\alpha}$, $\alpha< \aleph_1$,  are  likewise $ D^{({\aleph_1})}$-injective.
Let $B=D^{(\aleph_1)}$. Applying $\Hom_{R}(-,B)$ to $\mathcal{W}$, we obtain  the
commutative diagram as in (2.6) with each term $J_\alpha$ replaced by $\Hom_{R}(M_{\alpha},B)$.


We first show that $M\in {^{\perp}\textup{Cogen}_{*}(B)}$.
For each $\alpha< \aleph_1$, since $u_\alpha$ is $B$-injective, $M_\alpha\in {^{\perp}B}$.
Thus $M_\alpha$ is $B$-stationary by \cite[Example 3.13(1)]{A08}
and {$M_\alpha\in ^{\perp}\textup{Cogen}_{*}(B)$} by \cite[ Proposition 4.3]{Sa}.
By \cite[Lemma 2.2]{SS2}, $f_{\alpha+1,\alpha}$ is $\textup{Cogen}_{*}(B)$-injective, and
\cite[Lemma 2.4]{SS2} yields $M\in {^{\perp}\textup{Cogen}_{*}(B)}$.

\vspace{0.2cm}
Note that $\overline{B}$ is just the closure of
$\textup{Cogen}_{*}(B)$ under pure-epimorphic images.
Let $\varphi:E\rightarrow F$ be a pure-epimorphism
with $E\in \textup{Cogen}_{*}(B)$.
We must verify that every morphism $\psi:N\rightarrow F$ factorizes through $u$.
$${\xymatrix{
&&N\ar@{.>}[d]_{\phi}\ar[drr]|{\psi}\ar[rr]^{u}&&P\ar@{.>}[d] \\
0\ar[r]&\Ker(\varphi)\ar[r]&E\ar[rr]^{\varphi}&&F \ar[r]&0.}}$$
However, $\ker(\varphi)\in \textup{Cogen}_{*}(B)$ yields the existence of $\phi:N\rightarrow E$
such that $\varphi\phi=\psi$. Since $\coker(u)=M\in {^{\perp}\textup{Cogen}_{*}(B)}$ and $u$ is $B$-injective,
$u$ is $\textup{Cogen}_{*}(B)$-injective by the argument above.
We can then factorize the morphism $\phi$ through $u$. The
composition of the resulting map $\sigma: P\rightarrow  E$ with $u$ is the desired factorization.
This completes the proof.
\end{proof}

\section {\bf  Gorenstein projective modules and Gorenstein flat modules}

Recall that an exact sequence
$${\xymatrix{\cdots\ar[r]&P^{-2}\ar[r]&P^{-1}\ar[r]&P^{0}\ar[r]&P^{1}\ar[r]& \cdots}}\eqno {(3.1)}$$
of projective $R$-modules is a \emph{complete projective resolution} if $\Hom_{R}(-,P)$ leaves
it exact for every projective $R$-module $P$. An exact sequence $(3.1)$ of flat
$R$-modules is a \emph{complete flat resolution} if $E\bigotimes_{R} -$ leaves it exact for
every injective right $R$-module $E$.
An $R$-module $M$ is called \emph{Gorenstein
projective } (resp. \emph{Gorenstein flat}) \cite{EJa, EJT} if $M\cong \ker(P^{0}\rightarrow P^{1})$
in a complete projective (resp. complete flat) resolution $(3.1)$.
The class of Gorenstein projective modules is denoted by $\mathcal{GP}$, and
class of Gorenstein flat  modules is denoted by $\mathcal{GF}$.

A Gorenstein projective module $M$ is called \emph{strongly Gorenstein projective} \cite{BM} if there is a
complete projective resolution
$\cdots\rightarrow P\buildrel {f}\over\rightarrow P\buildrel {f}\over\rightarrow P\rightarrow \cdots$
with $M\cong \ker(f)$. Note that a module is Gorenstein projective if and only
if it is a direct summand of a strongly Gorenstein projective module (see \cite[Theorem 2.7]{BM}).

Holm \cite{HO1} asked whether every Gorenstein projective module is Gorenstein flat.
 It is well known (\cite[Lemma 3.2, Lemma 3.4]{WL}) that
each strongly countably presented Gorenstein projective (countably presented strongly Gorenstein projective) $R$-module  is Gorenstein flat.
The result we prove here can be seen as a continuation of this line.
The following  results give a  partial answer  to this question.

\begin{prop}\label{prop: 3.1}
Every $\aleph_1$-generated strongly Gorenstein projective module is Gorenstein flat.
\end{prop}
\begin{proof} Let $M$ be an $\aleph_1$-generated strongly Gorenstein projective module. We  have $M\in {^\perp \overline{R}}$
by Proposition \ref{prop: 2.4}  and thus $M$ is strict $R$-stationary by \cite[Lemmas 4.2]{Sa}. Consequently, the result
follows from \cite[Theorem 2.2]{EM}.
\end{proof}

Let $\kappa$ be an infinite cardinal. An $R$-module $M$ is said to be \emph{strongly $\kappa$-presented} if
it has a projective resolution consisting of $\kappa$-generated projective modules.
	\begin{prop}\label{s1GPGF}
Every strongly $\aleph_1$-presented Gorenstein projective module is Gorenstein flat.
	\end{prop}
	
	\begin{proof}
	Let $M$ be a strongly $\aleph_1$-presented Gorenstein projective $R$-module. By \cite[Proposition 2.4]{HO}, there exists a short exact sequence $0 \to M \to L \to M_1 \to 0$ with $L$ free and $M_1$ Gorenstein projective. Since $M$ is $\aleph_1$-generated, there is a set $\{m_i\}$ of cardinal $\aleph_1$ that generates $M$. Without loss of generality, we assume that $M$ is a submodule of $L$. Suppose that $\{l_j\}$ is a set of bases of the free $R$-module $L$. Since each $m_i$ is a finite $R$-linear combination of some elements from $\{l_j\}$, there is subset of $\{l_j\}$ of cardinal $\aleph_1$ which generates a free left $R$-module $L^0$ such that $L^0$ is a direct summand of $L$ and $M$ is a submodule of $L^0$. By the proof of \cite[Lemma 2.9]{MTY}, we know that $L^0/M$ is isomorphic to a direct summand of $M_1$, and so $L^0/M$ is Gorenstein projective. Repeating this procedure, we get an exact sequence
	\[
	0 \longrightarrow M \longrightarrow L^0 \xrightarrow{d^0} L^1 \longrightarrow \cdots
	\]
	with each $L^i$ a $\aleph_1$-generated projective $R$-module and each $\operatorname{im}(d^i)$ a Gorenstein projective $R$-module. On the other hand, $M$ is strongly $\aleph_1$-presented, so we get an exact sequence of $R$-modules
	\[
	\cdots \longrightarrow L_1 \xrightarrow{d_1} L_0 \xrightarrow{d_0} L_0 \longrightarrow M \longrightarrow 0
	\]
	with each $L_i$ $\aleph_1$-generated and projective. We notice that each $\operatorname{im}(d_i)$ is Gorenstein projective. Now assembling the above two sequences, we get a complete projective resolution
	\[
	\cdots \longrightarrow L_1 \xrightarrow{d_1} L_0 \xrightarrow{d_0} L^0 \xrightarrow{d^0} L^1 \longrightarrow \cdots
	\]
	such that each $L_i$ and $L^i$ are $\aleph_1$-generated for $i \geq 0$ and $M \cong \ker(L^0 \to L^1)$.
	
	 Using a similar proof as proved in \cite[Theorem 2.7]{BM}, we get that $M$ is isomorphic to a direct summand of a strongly $\aleph_1$-presented strongly Gorenstein projective left $R$-module $N$.	 Note that $N$ are Gorenstein flat by Proposition \ref{prop: 3.1}. Then all left $R$-modules in $^\perp(N^\perp)$ are Gorenstein flat by \cite[Proposition 2.6]{WL}. Thus $M$ is Gorenstein flat since $M \in ^\perp(N^\perp)$.
	\end{proof}
 Let $R$ be a ring. An $R$-module $M$  is said to be \emph{$\kappa$-presented} if there exists an exact sequence
$0 \rightarrow K \rightarrow F \rightarrow M \rightarrow 0$ where $F$ is a free module of rank $\kappa$ and $K$ is a $\kappa$-generated submodule of $F$. A ring $R$ is called a \emph{left $\kappa$-coherent ring} if every $\kappa$-generated left ideal $I$ of $R$ is $\kappa$-presented. It is easy to verify that  every $\kappa$-presented $R$-module is strongly  $\kappa$-presented over a left $\kappa$-coherent ring $R$, where $\kappa$ is an infinite regular cardinal.

\begin{cor}
 Let $R$ be a {left} $\aleph_1$-coherent ring. Then every $\aleph_1$-generated  Gorenstein projective $R$-module is Gorenstein flat.
\end{cor}
\begin{proof} It follows by Propositions \ref{s1GPGF}.
\end{proof}

 In \cite[Section 4]{SS2},
\v{S}aroch and \v{S}\'{t}ov\'{\i}\v{c}ek introduced a class of \emph{projectively coresolved Gorenstein flat modules}
and denote this class by $\mathcal{PGF}$.
They are the cycles of exact sequences of projective modules that remain exact when tensored with any injective right $R$-module. It is obvious that $\mathcal{FL}\bigcap \mathcal{GF}=\mathcal{FL}$ and
$\mathcal{P}\subseteq \mathcal{FL} \bigcap \mathcal{GP}$. It is well known that
\cite[Theorem 4.11]{SS2}, $\mathcal{GF}\bigcap \mathcal{PGF}^{\perp}=\mathcal{FL}$ and $\mathcal{PGF}\subseteq\mathcal{GP}$.
 A natural question is to determine whether
the two classes coincide, i.e., $\mathcal{P}=\mathcal{FL} \bigcap \mathcal{GP}$, equivalent to
that \cite[Question 3.8]{BA} whether any flat Gorenstein projective module is projective.
The main concern in this section is to discuss when a flat Gorenstein projective module is projective.
Regarding this question, there are some known results as follows.
\begin{itemize}
\item Every flat Gorenstein projective $R$-module is trivially projective when $R$ is perfect.
\item \cite[Corollary 2.7]{BA} Every  flat and strongly Gorenstein projective module is projective.
\item \cite[Proposition 2.9]{BA} If $\textup{FPD$_{Flat}$}(R)<\infty$, then every flat Gorenstein projective module is projective.
\end{itemize}

Moreover, we have the following results.
\begin{lem}\label{gpgfplfgp}
If $\mathcal{GP}\subseteq \mathcal{GF}$, then $\mathcal{P}=\mathcal{FL} \bigcap \mathcal{GP}$.
\end{lem}
\begin{proof} It follows by the assumption and \cite[Theorem 4.11]{SS2} that $$\mathcal{FL} \bigcap \mathcal{GP} =(\mathcal{GF}\bigcap \mathcal{PGF}^{\perp})\bigcap \mathcal{GP}= \mathcal{PGF}^{\perp}\bigcap \mathcal{GP}= \mathcal{GP}^{\perp}\bigcap \mathcal{GP}=\mathcal{P}.$$

\end{proof}






\begin{prop} \label{prop: 6.1}
Let $(M_{i})_{i\in I}$ be a direct system of projective modules such that $|I|<\aleph_{\omega}$ and $M=\lim M_{i}$.
If $M$ is $($flat$)$ Gorenstein projective, then it is  projective.
\end{prop}
\begin{proof}
By \cite[Lemma 6.5]{SS1}, there is an exact sequence:
$0\rightarrow X_{n}\rightarrow \cdots \rightarrow X_{1}\rightarrow X_{0}\rightarrow \lim M_{i}\rightarrow 0,$
where $n$ is a non-negative integer, $X_{j} \in \textup{Add} \{M_{i} \mid i \in I \}$ and $X_{j}$ is projective for all $j= 0,~\cdots,~n$.
Thus $M$ is projective.
\end{proof}

At the end of this section, we give a sufficient condition on when  Gorenstein projective modules are Gorenstein flat.
The set difference of $A$ and $B$ is denoted as $A-B$.

\begin{prop} \label{prop: 4.8} Let $R$ be a ring. If $\mathcal{PGF}^{\perp}$ is closed under direct sums and
$\mathcal{PGF}^{\perp}=\mathcal{GP}^{\perp}\bigcup \mathcal{GF}^{\perp}$,
then every Gorenstein projective module is Gorenstein flat.
In particular, every flat Gorenstein projective module is projective in this case.
\end{prop}
\begin{proof} Suppose $R$ is a perfect ring. Since $\mathcal{PGF}=\mathcal{GF}\subseteq\mathcal{GP}$, it follows that
${\mathcal{GP}^{\perp}}\subseteq\mathcal{PGF}^{\perp}$, and hence $\mathcal{GP}=\mathcal{PGF}$
by \cite[Corollary 3.4(1)]{CO}. For the non-perfect case,
we will prove that
$\mathcal{GP}\bigcap \mathcal{PGF}^{\perp}=\mathcal{P}$. Clearly,  $\mathcal{GP}\bigcap \mathcal{PGF}^{\perp}\supseteq\mathcal{P}$. On the other hand,
pick any nonzero $M\in\mathcal{GP}\bigcap\mathcal{PGF}^{\perp}$. Since $M\in \mathcal{PGF}^{\perp}=\mathcal{GP}^{\perp}\bigcup \mathcal{GF}^{\perp}$, we can divide it into two cases:

\textit{Case 1}: If $M\in \mathcal{GP}^{\perp}$, then $M\in \mathcal{P}$ since $\mathcal{GP}\bigcap\mathcal{GP}^{\perp}=\mathcal{P}$.

\textit{Case 2}: If $M\in \mathcal{GF}^{\perp}-\mathcal{GP}^{\perp}$, then any direct sum $M^{(I)}\in \mathcal{GF}^{\perp}-\mathcal{GP}^{\perp}$. Otherwise, by assumption,  $M^{(I)}\in\mathcal{GP}\bigcap\mathcal{GP}^{\perp}=\mathcal{P}$, which implies that $M$
is projective and thus $M\in \mathcal{GP}^{\perp}$, contradicting the choice of $M$. Now let
$${\xymatrix{0\ar[r]&N\ar[r]&R^{(X)}\ar[r]&M\ar[r]& 0}}$$
be a free presentation of $M$. It is clear that $N\in\mathcal{GP}\bigcap\mathcal{PGF}^{\perp}$, and thus
$N\in\mathcal{GF}^{\perp}-\mathcal{GP}^{\perp}$.
If not, $N\in \mathcal{GP}\bigcap\mathcal{GP}^{\perp}=\mathcal{P}$, which would imply that
the above exact sequence splits and
$M\in \mathcal{P}$, this contradicting $M\in \mathcal{GF}^{\perp}-\mathcal{GP}^{\perp}$.
Note that any direct sum of $N$ and $M$ lies in $\mathcal{GF}^{\perp}-\mathcal{GP}^{\perp}$.
We conclude that
$R^{(X)}\in \mathcal{GF}^{\perp}$ and thus is a $\sum$-cotorsion module.
By \cite[Theorem 3.3]{SS2}, every module in $\overline{R}$ is cotorsion, and in particular, every flat module is cotorsion. It follows by \cite[Proposition 3.3.1]{X96} that  $R$ is perfect, which  contradicts to the assumption that $R$ is not perfect.
Therefore, $\mathcal{GP}\bigcap \mathcal{PGF}^{\perp}=\mathcal{P}$.

Now let $M\in \mathcal{GP}$, there exists an epimorphic $\mathcal{PGF}$-precover $\varphi:Q\rightarrow M$ with
$\ker(\varphi)\in \mathcal{GP}\bigcap \mathcal{PGF}^{\perp}=\mathcal{P}$ (see \cite[Theorem 4.9]{SS2}). It follows that $M$ is a direct summand of the $PGF$-module $Q$, and
thus every Gorenstein projective module is Gorenstein flat. It follows by Lemma \ref{gpgfplfgp} that every flat Gorenstein projective module is projective in this case.
\end{proof}

\section{\bf Homological invariants of $\aleph_1$-coherent rings}

In the past few years, many algebraists have investigated various of homological invariants, including
\begin{itemize}
\item $\spli R$: the supremum of the projective lengths of injective left $R$-modules;
\item $\silp R$: the supremum of the injective lengths of projective left $R$-modules;
\item $\sfli R$: the supremum of the flat lengths of injective left $R$-modules.
\end{itemize}
In the same way, one may consider right $R$-modules and define the right invariants $\spli R^{\op}$, $\silp R^{\op}$, and $\sfli R^{\op}$ respectively. The relation between $\spli R$ and $\silp R$ is unclear for a general ring $R$ and ask whether these two invariants are always equal.  In the special case where $R$ is an Artin algebra, the equality $\spli R = \silp R$ is equivalent to  the so-called Gorenstein Symmetry Conjecture.

It is well-known that the equivalence of finiteness of both $\spli R$ and $\silp R$ implies that $\spli R = \silp R$. It was proved in  \cite{CE25} that when both $\spli R$ and $\silp R$ are finite, then  $\spli R^{\mathrm{op}}=\silp R^{\mathrm{op}}$. So we turn to investigate when $$\silp R + \silp R^{\mathrm{op}} = \spli R + \spli R^{\mathrm{op}}~~?$$ since $\spli R = \silp R$ whenever $R\cong R^{\mathrm{op}}$ in this case.  In this section, we will show that the equality $\silp R + \silp R^{\mathrm{op}} = \spli R + \spli R^{\mathrm{op}}$ always holds for two-sided $\aleph_1$-coherent rings, which is a new and nontrivial result.

\begin{prop}\label{1sflisilp}
 Let $R$ be a {left} $\aleph_1$-coherent ring. Then $\sfli R^{\mathrm{op}}\leq \silp R$.
\end{prop}
\begin{proof}
Let $R$ be a {left} $\aleph_1$-coherent ring. Assume first that $\operatorname{silp} R = \infty$; in that case the inequality holds trivially. Now suppose $\operatorname{silp} R = m < \infty$, and let $J$ be an arbitrary finitely generated ideal of $R$. Since $\operatorname{silp} R = m$, we have $\operatorname{Ext}_R^k(R/J,R^{(\aleph_1)})=0$ for all $k \ge m+1$. Hence $\operatorname{Ext}_R^n\bigl(\Omega^m(R/J),R^{(\aleph_1)}\bigr)=0$ for every $n \ge 1$, where $\Omega^m(R/J)$ denotes the $m$-th syzygy of $R/J$.

Because $R$ is {left} $\aleph_1$-coherent, the finitely presented $R$-module $R/J$ is strongly $\aleph_1$-presented; consequently $\Omega^m(R/J)$ is also strongly $\aleph_1$-presented. It then follows from Proposition \ref{prop: 2.4} that
$$
\operatorname{Ext}_R^n\bigl(\Omega^m(R/J),B\bigr)=0 \qquad (n\ge 1)
$$
for every module $B$ in $\overline{R}$, the definable closure of $R$. The class $\overline{R}$ is closed under direct limits and direct products, so by Lemma \ref{ML} the module $\Omega^m(R/J)$ is strict $\overline{R}$-Mittag-Leffler. Applying Proposition \ref{nSML} we obtain that the natural homomorphism
$$
\Phi_{\Omega^{m}(R/J)}^{(1)} \colon \Tor_{1}^{R}\!\bigl(\Hom_{\mathbb{Z}}(B,G),\, \Omega^{m}(R/J)\bigr) \longrightarrow \Hom_{\mathbb{Z}}\!\bigl(\Ext_{R}^{1}(\Omega^{m}(R/J),B), G\bigr)
$$
is injective for all $B\in\overline{R}$ and every divisible abelian group $G$.

Since $\Ext_R^1(\Omega^{m}(R/J),B)=0$, the above implies
$$
\Tor_{1}^{R}\!\bigl(\Hom_{\mathbb{Z}}(B,G),\, \Omega^{m}(R/J)\bigr)=0,
$$
and so
$$
\Tor_{m+1}^{R}\!\bigl(\Hom_{\mathbb{Z}}(R,G),\, R/J\bigr)=0.
$$
for every finitely generated ideal $J$ of $R$, every $B\in\overline{R}$ and every divisible abelian group $G$.
Therefore $\fd_{ R^{\mathrm{op}}}\Hom_{\mathbb{Z}}(R,G)\leq m$ for every finitely generated ideal $J$ of $R$, every $B\in\overline{R}$ and every divisible abelian group $G$.  Every injective right $R$-module $I$ is a direct summand of a module of the form $\Hom_{\mathbb{Z}}(R,G)$ for a suitable divisible abelian group $G$. Consequently,
$$\sfli R^{\mathrm{op}} \leq m=\operatorname{silp} R.$$
\end{proof}

\begin{prop}\label{3.2}
If $R$ is a {left} $\aleph_1$-coherent ring, then we have inequalities
$$
\spli R^{\mathrm{op}} \leqslant \silp R + \silp R^{\mathrm{op}} \leqslant \silp R + \spli R + \spli R^{\mathrm{op}}.
$$
\end{prop}

\begin{proof}
The first inequality in the statement follows by
$$
\spli R^{\mathrm{op}} \leqslant \sfli R^{\mathrm{op}} + \silp R^{\mathrm{op}} \leqslant \silp R + \silp R^{\mathrm{op}}.
$$
In the above chain of inequalities, the first one follows by applying \cite[Proposition 2.2]{E11} to the ring $R^{\mathrm{op}}$, whereas the second one is a consequence of Proposition \ref{1sflisilp}.

The second inequality in the statement follows by
$$
\silp R^{\mathrm{op}} \leqslant \sfli R + \spli R^{\mathrm{op}} \leqslant \spli R + \spli R^{\mathrm{op}}.
$$
for any ring $R$. In the above chain of inequalities, the first one follows by applying \cite[Corollary 22]{D23} to the ring $R^{\mathrm{op}}$, whereas the second one follows by the flat dimension of any module is always bounded by its projective dimension.
\end{proof}

\begin{thm}\label{silp+}
Let $R$ be a ring which is both left and right $\aleph_1$-coherent. Then the following conditions are equivalent:
\begin{enumerate}
\item[(i)] The invariants $\silp R$ and $\silp R^{\mathrm{op}}$ are finite.
\item[(ii)] The invariants $\spli R$ and $\spli R^{\mathrm{op}}$ are finite.
\end{enumerate}
If these conditions are satisfied, then
$$
\silp R = \spli R < \infty \quad \text{and} \quad \silp R^{\mathrm{op}} = \spli R^{\mathrm{op}} < \infty.
$$
Thus, we always have an equality $\silp R + \silp R^{\mathrm{op}} = \spli R + \spli R^{\mathrm{op}}$.
\end{thm}

\begin{proof}
(i) $\Rightarrow$ (ii): Since the assumption on $R$ and assertion (i) are both left-right symmetric, it suffices to show that $\spli R^{\mathrm{op}}$ is finite. This finiteness claim follows from the first inequality in Proposition \ref{3.2}.

(ii) $\Rightarrow$ (i): This follows from \cite[Proposition 23]{D23}, which is valid without any assumptions on the ring $R$.

The others follow by the equivalence of  (i) and (ii), and that either one implies that $\silp R=\spli R$ and $\silp R^{\mathrm{op}}=\spli R^{\mathrm{op}}$.
\end{proof}

\begin{cor}
Let $R$ be a ring which is isomorphic with its opposite $R^{\mathrm{op}}$. If $R$ is left $($and hence right$)$ $\aleph_1$-coherent, then $\silp R = \spli R$.
\end{cor}

\begin{cor}
If $R$ is a commutative $\aleph_1$-coherent ring, then $\silp R = \spli R$.
\end{cor}


Recall that a ring $R$ is said to be a left generalized coherent ring if the supremum of flat dimensions of all level right $R$-modules (i.e., right $R$-modules $M$ satisfying $\Tor^R_1 (M, T) = 0$ for any strongly finitely presented left $R$-module $T$) is finite; to be
left weakly coherent ring if any direct product of flat right $R$-modules has finite flat dimension. It follows by \cite[Corollary 2.11]{BGH14} that every  left coherent is left generalized coherent.
 Since the direct product of flat modules is level, one has that every left generalized coherent ring is left  weak coherent.

In 2011,  Emmanouil and Talelli implied that the above equation holds for two-sided $\aleph_0$-Noetherian rings (see \cite[Theorem 3.7]{E11}). In 2025,  Chatzistavridis and Emmanouil obtained that the above equation holds for two-sided weakly coherent rings (see \cite[Theorem 3.4]{CE25}). Subsequently,  Wang and Yang Suggested that the above equation holds for two-sided generalized coherent rings  (see \cite[Theorem 3.11]{WY25}).
 In the rest of this section, we will construct a class of $\kappa$-coherent rings which is not weakly coherent, and so also is not generalized coherent making our Theorem \ref{silp+} is new and non-trivial.

\begin{prop}\label{thm:main}
Let $(R,\m,k)$ be a commutative local ring with $\m^2=0$, and suppose $\dim_k\m=\kappa$ is infinite. Then $R$ is  a $\kappa$-coherent ring.
Set $M = R^{\aleph}$ the direct product of $\aleph$-copies of $R$, where ${\aleph}$ is an arbitrary infinite cardinal.
Then
$$
\fd_R(M)=\infty.
$$
Consequently, $R$ is not a weakly coherent ring, and so not a generalized ring.
\end{prop}
\begin{proof}
	Since $\m^2=0$, $\m$ is the unique prime ideal. And since  $\dim_k\m=\kappa$, then $\m$ can be generated by $\kappa$ elements. And so $R$ is a $\kappa$-Noetherian ring, i.e., every ideal of $R$ can be generated by $\kappa$ elements. So it is easy to verify that  $R$ is also $\kappa$-coherent.
	
	Let $M =R^{\aleph}$. Take mutually distinct coordinates $i_1, i_2, \dots$ in $\aleph$, and take linearly independent elements $v_1, v_2, \dots$ in $\m$. Define $y = (y_i)_{i \in {\aleph}} \in \m^{\aleph}$:
	$$y_{i_n} = v_n, \quad y_i = 0 \quad (i
	\notin \{i_n\mid n=1,2,\dots\}).$$
	
	Claim that $y \not\in \m M$. Otherwise,
	$$y = \sum_{j=1}^t a_j x^{(j)}, \quad a_j \in \m, x^{(j)} \in M.$$
	In each coordinate $i$, since $m^2 = 0$, we have
	$$y_i = \sum_{j=1}^t a_j x_{i}^{(j)} \in \operatorname{span}_k\{a_1, \dots, a_r\}.$$	
	This forces all $v_i$ to lie in the same finite-dimensional subspace, which is a contradiction.
	
	Since $\m$ is killed by multiplication, there is a natural isomorphism
	$$\m \otimes_R M \cong \m \otimes_R (M/\m M).$$	
	Take $0
	\not= a \in \m$. Because $a
	\not= 0$ and $\bar{y}
	\not= 0$, the tensor
	$a \otimes_k \bar{y}$
	is non-zero, and so the tensor 	$$a \otimes_R y\not=0.$$  But its image under the multiplication map
	$$\m \otimes_R M \to M$$
	is $ay = 0$. Tensoring the following sequence with $M$
$$0\rightarrow\m\rightarrow R\rightarrow k\rightarrow 0,$$	
	we get
	$$0 \to \operatorname{Tor}_1^R(k, M) \to \m \otimes_R M \to M,$$
	so $\operatorname{Tor}_1^R(k, M)
	\not= 0$.
	
	Furthermore, because $\m \cong k^{(\kappa)}$ as an $R$-module, for $n \geq 1$ by dimension shift we have
	$$\operatorname{Tor}_{n+1}^R(k, M) \cong \operatorname{Tor}_n^R(\m, M) \cong \operatorname{Tor}_n^R(k, M)^{(\kappa)}.$$
	By induction, we obtain
	$$\operatorname{Tor}_n^R(k, M)
	\not= 0 \quad (n \geq 1).$$
	
	Therefore $\operatorname{fd}_R M = \infty$. And so $R$ is not weakly coherent; generalized coherent implies weakly coherent, and the last item follows.
\end{proof}

Finally, we give a class of rings satisfying the hypotheses of Proposition~\ref{thm:main}.

\begin{exa}\label{ex:poly}
Let $k$ be a field and let
$$
R = \frac{k[x_i\mid i\in\aleph_1]}{(x_ix_j \mid i,j\in\aleph_1)}.
$$
Write $\bar{x_i}$ for the images of the variables.
The maximal ideal $\m=(\bar{x_i}\mid i\in\aleph_1)$ satisfies $\m^2=0$, and $R/\m\cong k$.
As a $k$-vector space, $\m$ has basis $\{\bar{x_i}\mid i\in\aleph_1\}$, so $\dim_k\m=\aleph_1$.
Then $R$ is an $\aleph_1$-coherent ring but is neither weakly coherent nor generalized coherent.
\end{exa}

\begin{proof}
The relation $\bar{x_i}\bar{x_j}=0$ for all $i,j$ gives $\m^2=0$.
Linear independence of the $\bar{x_i}$: if $\sum a_i \bar{x_i}=0$ in $R$, then the corresponding polynomial belongs to the ideal generated by all $x_ix_j$, but that ideal contains only polynomials whose terms have degree at least $2$, while $\sum a_ix_i$ is homogeneous of degree $1$; hence all $a_i=0$. Hence $R$ is not a weakly coherent ring, and so is not generalized coherent. Since $\m$ can not be generated by countable elements, $R$ is not $\aleph_0$-Noetherian. However, every ideal of $R$ can be generated by $\aleph_1$ elements, and so $R$ is an $\aleph_1$-Noetherian (and hence an $\aleph_1$-coherent) ring.
\end{proof}



\end{document}